\documentclass[11pt,reqno]{amsart}
\usepackage{amssymb,latexsym}
\usepackage{graphicx}

\usepackage{cite}
\usepackage{url}

\calclayout

\usepackage[utf8]{inputenc}
\usepackage[T1]{fontenc}
\usepackage{amsfonts}
\usepackage[greek,english]{babel}
\usepackage{lmodern}

\usepackage{enumerate}
\usepackage{enumitem}
\usepackage{chngcntr} % for: \counterwithout

\usepackage[normalem]{ulem} % for strike out: \sout

\theoremstyle{plain}
\numberwithin{equation}{section}
\newtheorem{thm}{Theorem}[section]
\newtheorem{theorem}[thm]{Theorem}
\newtheorem{lemma}[thm]{Lemma}

\newtheorem{remark}[thm]{Remark}

\counterwithout{equation}{section}
\begin{document}
\setcounter{page}{1}

\title{Generalization of a Ramanujan identity}
\author{Örs Rebák}
\address{Independent scholar\\
               Győr, Hungary}
\email{rebakors@gmail.com}
%\thanks{Thanks.}

\begin{abstract}
The Euler product for the Landau--Ramanujan constant could have motivated a curious identity by Ramanujan that appears in his notebooks two times. This observation involves a square root and the first four prime numbers of the form $4n+3$, i.e., $3, 7, 11, 19$. Berndt asks whether Ramanujan's identity is an isolated result, or if there are other identities of this type. With this work we would like to give a possible answer to Berndt's question.
\end{abstract}

%\date{22 December 2016}

\maketitle

\section{Introduction}

Let $B(x)$ denote the number of positive integers not exceeding $x$ that can be expressed as a sum of two squares. Landau \cite{Landau}, \cite[pp.~641--669]{Landau2} in 1908 showed that
\[
B(x) \sim K \frac{x}{\sqrt{\log x}} \text{ \, as } x \to \infty,
\]
where $K$ is a constant. Independently, Ramanujan in his first letter to Hardy \cite[pp.~52 and 60--62]{BerndtIV}, \cite[p.~24]{BerndtRankin}, \cite[p.~xxiv]{Hardy} in 1913 stated the following. The number of numbers greater than $A$ and less than $x$ that can be expressed as a sum of two
squares is
\[
K \int_A^x \frac{dt}{\sqrt{\log t}} + \theta(x),
\]
where $K= 0.764\dots$ and $\theta(x)$ is very small when compared with the previous integral. This statement also appears in Ramanujan's second and third notebooks \cite[pp.~307 and 363]{Ramanujan}.\\
\indent
Since then the quantity $K$ has been known as the Landau--Ramanujan constant \cite[pp.~98--104]{Finch}. An exact formula for $K$ is given by its Euler product expansion
\[
K = \frac{1}{\sqrt{2}} \prod_{p} \left(\frac{1}{1-1/p^2}\right)^{1/2},
\]
where $p$ runs through the primes of the form $4n+3$. This could have motivated the following observation that appears in Ramanujan's notebooks \cite[pp.~309 and 363]{Ramanujan} two times:
\begin{equation}\label{eq:ri}
\sqrt{2\left(1-\frac{1}{3^2}\right)\left(1-\frac{1}{7^2}\right)\left(1-\frac{1}{11^2}\right)\left(1-\frac{1}{19^2}\right)} = \left(1+\frac{1}{7}\right)\left(1+\frac{1}{11}\right)\left(1+\frac{1}{19}\right).
\end{equation}
``It may be somewhat interesting to note,'' Ramanujan wrote about \eqref{eq:ri} in one of his letters to Hardy \cite[p.~177]{BerndtRankin}, referring to the fact that the squared numbers on the left-hand side are the first four prime numbers of the form $4n+3$. We examine the identity~\eqref{eq:ri} that can be found in Berndt's book \cite[p.~20]{BerndtIV} and also in the Andrews--Berndt book \cite[pp.~410--411]{AndrewsBerndt} with only a brief discussion.

\section{Generalization of Ramanujan's identity}

Berndt \cite[p.~20, Entry~6]{BerndtIV} asks whether Ramanujan's identity in \eqref{eq:ri} is an isolated result, or if there are other identities of this type. The result of this section is a possible answer to Berndt's question.

\begin{lemma}\label{lemma-1} Let $n \geq m \geq 1$ be integers, and let $\left(a_k\right)$ be a sequence of real numbers such that $a_k \neq 0,1$ for all $k=m,\dots,n$ and $-1 < a_\ell < 0$ for an even number of $\ell \in \{m,\dots,n\}$ indices. Then
\begin{equation}\label{eq:lemma}
\sqrt{\prod_{k=m}^n \frac{a_k+1}{a_k-1} \cdot \prod_{k=m}^n \left(1-\frac{1}{a^2_k}\right)} = \prod_{k=m}^n \left(1+\frac{1}{a_k}\right).
\end{equation}
\end{lemma}
\begin{proof}
Because of the condition on the elements of $(a_k)$, the expression under the square root and the right-hand side are nonnegative. Both sides of \eqref{eq:lemma} are equal to zero if and only if $a_k = -1$ for some $k \in \{m,\dots,n\}$. Suppose that $a_k \neq -1$ for all $k=m,\dots,n$. Note that for a given real number $a \neq 0,\pm 1$, we have
\begin{equation}\label{eq:obs}
 \frac{a+1}{a-1} = \frac{1+\frac{1}{a}}{1-\frac{1}{a}} = \frac{\left(1+\frac{1}{a}\right)^2}{\left(1-\frac{1}{a}\right)\left(1+\frac{1}{a}\right)} =  \frac{\left(1+\frac{1}{a}\right)^2}{1-\frac{1}{a^2}}.
\end{equation}
By using this observation, straightforward arithmetic gives the result.
\end{proof}

Henceforth we use Lemma~\ref{lemma-1} to deduce such identities of the form of \eqref{eq:lemma}, which have a closed-form expression for the product $\prod_{k=m}^n \frac{a_k+1}{a_k-1}$.

\begin{theorem}[Generalization of Ramanujan's identity]\label{genr} Let $a \in \left(-\infty,-2/3\right) \cup \left(-1/2,-1/3\right] \cup \left(\left(-1/6,\infty\right) \setminus \{0,1\}\right)$ be a real number. Then
\begin{multline}\label{eq:gi}
\sqrt{\frac{a+1}{a-1}\left(1-\frac{1}{a^2}\right)\left(1-\frac{1}{(2a+1)^2}\right)\left(1-\frac{1}{(3a+2)^2}\right)\left(1-\frac{1}{(6a+1)^2}\right)}
\\ = \left(1+\frac{1}{2a+1}\right)\left(1+\frac{1}{3a+2}\right)\left(1+\frac{1}{6a+1}\right).
\end{multline}
\end{theorem}

By substituting $a=3$ into \eqref{eq:gi}, we arrive at Ramanujan's identity~\eqref{eq:ri}.

\begin{proof}
Suppose that $a \neq -1/3$. We can use Lemma~\ref{lemma-1} with $(a_k) = (2a+1,3a+2,6a+1)$. For the first product of \eqref{eq:lemma}, we find that
\[
 \prod_{k=1}^3 \frac{a_k+1}{a_k-1} = \frac{\left(2a+1\right)+1}{\left(2a+1\right)-1} \cdot  \frac{\left(3a+2\right)+1}{\left(3a+2\right)-1} \cdot  \frac{\left(6a+1\right)+1}{\left(6a+1\right)-1} = \frac{\left(a+1\right)^2}{a^2}.
\]
On the other hand, if we suppose that $a \neq -1$, by using \eqref{eq:obs}, we have
\[
\frac{a+1}{a-1}\left(1-\frac{1}{a^2}\right)  = \frac{\left(1+\frac{1}{a}\right)^2}{1-\frac{1}{a^2}}\left(1-\frac{1}{a^2}\right)  = \frac{\left(a+1\right)^2}{a^2}.
\]
Since both sides of \eqref{eq:gi} are equal to zero if and only if $a=-1$ or $a=-1/3$, the proof is complete.
\end{proof}

\begin{remark}[Alternative form]\label{agenr}
 It is clear from the proof of Theorem~\ref{genr} that the following identity holds. Let $a \in \left(-\infty,-2/3\right) \cup \left(-1/2,-1/3\right] \cup \left(\left(-1/6,\infty\right) \setminus \{0\}\right)$ be a real number. Then
\begin{multline}\label{eq:gialt}
\frac{a+1}{a}\cdot\sqrt{\left(1-\frac{1}{(2a+1)^2}\right)\left(1-\frac{1}{(3a+2)^2}\right)\left(1-\frac{1}{(6a+1)^2}\right)}
\\ = \left(1+\frac{1}{2a+1}\right)\left(1+\frac{1}{3a+2}\right)\left(1+\frac{1}{6a+1}\right).
\end{multline}
\end{remark}

We can derive similar identities by using Lemma~\ref{lemma-1} with the sequence $(a_k) = (2a+1,\linebreak 3a+1,6a+5)$ or with $(a_k) = (2a+1,4a+1,4a+3)$ with a suitable condition on $a$. By substituting $a=1$ into \eqref{eq:gialt}, we find that
\begin{equation}\label{eq:oddi}
2\cdot\sqrt{\left(1-\frac{1}{3^2}\right)\left(1-\frac{1}{5^2}\right)\left(1-\frac{1}{7^2}\right)} = \left(1+\frac{1}{3}\right)\left(1+\frac{1}{5}\right)\left(1+\frac{1}{7}\right).
\end{equation}
We generalize \eqref{eq:oddi} for further odd denominators in \eqref{eq:thm-1-3} of Theorem~\ref{thm-1}.

\section{Identities with telescoping products}

In this section we use Lemma~\ref{lemma-1} with appropriate $(a_k)$ sequences, for which the product $\prod_{k=m}^n \frac{a_k+1}{a_k-1}$  has a telescoping property.

\begin{theorem}\label{thm-1} Let $n$ and $m$ be integers. For $n\geq m \geq 2$, we have
\begin{equation}\label{eq:thm-1-1}
\sqrt{ \frac{n(n+1)}{m(m-1)} \cdot \prod_{k=m}^n \left(1- \frac{1}{k^2} \right)} = \prod_{k=m}^n \left(1+\frac{1}{k}\right).
\end{equation}
For $n\geq m \geq 1$, we have
\begin{equation}\label{eq:thm-1-2}
\sqrt{ \frac{2n+1}{2m-1} \cdot \prod_{k=m}^n \left(1-\frac{1}{(2k)^2}\right)} = \prod_{k=m}^n\left(1+\frac{1}{2k}\right)
\end{equation}
and
\begin{equation}\label{eq:thm-1-3}
\sqrt{ \frac{n+1}{m} \cdot \prod_{k=m}^n \left(1-\frac{1}{(2k+1)^2}\right) } = \prod_{k=m}^n \left(1+\frac{1}{2k+1}\right).
\end{equation}

\end{theorem}

Note that \eqref{eq:thm-1-1} gives $\left(n+1\right)/m$, which appears under the square root in \eqref{eq:thm-1-3}.

\begin{proof}
In order to prove \eqref{eq:thm-1-1}, according to Lemma~\ref{lemma-1}, we have to show the closed-form of a telescoping product. We find that
\begin{align*}
\prod_{k=m}^n \frac{k+1}{k-1} &= \frac{m+1}{m-1} \cdot \frac{m+2}{m} \cdot \frac{m+3}{m+1} \cdot \, \cdots \, \cdot \frac{n-1}{n-3} \cdot \frac{n}{n-2} \cdot \frac{n+1}{n-1} \\ &= \frac{n(n+1)}{m(m-1)}.
\end{align*}
The proofs of \eqref{eq:thm-1-2} and \eqref{eq:thm-1-3} are analogous.
\end{proof}
\pagebreak
Berndt's question may have other interesting answers. It would be worth examining Lemma~\ref{lemma-1} further with various $(a_k)$ sequences. In the following theorem, we use $\left(a_k\right)=\left(k^3\right)$.

\begin{theorem} Let $n \geq m \geq 2$ be integers. Then
\[
\sqrt{\frac{m(m-1)+1}{m(m-1)} \cdot \frac{n(n+1)}{n(n+1)+1} \cdot \prod_{k=m}^{n} \left(1-\frac{1}{k^6}\right)} = \prod_{k=m}^{n} \left(1+\frac{1}{k^3}\right).
\]
\end{theorem}

\begin{proof}
According to Lemma~\ref{lemma-1}, we have to deduce the following.
\begin{alignat*}{2}
\prod_{k=m}^n \frac{k^3+1}{k^3-1} &= \frac{m^3+1}{\left(n+1\right)^3+1} \cdot \prod_{k=m}^n \frac{\left(k+1\right)^3+1}{k^3-1} &&= \frac{m^3+1}{\left(n+1\right)^3+1} \cdot \prod_{k=m}^n \frac{k+2}{k-1} \\
&= \frac{m^3+1}{\left(n+1\right)^3+1} \cdot \frac{n(n+1)(n+2)}{(m-1)m(m+1)} &&= \frac{m(m-1)+1}{m(m-1)} \cdot \frac{n(n+1)}{n(n+1)+1}.\tag*{\qedhere}
\end{alignat*}
\end{proof}

\medskip

\subsubsection*{Acknowledgments} The author would like to thank the referees for their valuable suggestions, which greatly improved the presentation of the paper. The comments of Sándor Bozóki (MTA SZTAKI) and Máté Horváth (CrySyS Lab) are highly appreciated.

\medskip
%\noindent MSC2010.
\end{document}